\documentclass[12pt]{amsart}
\usepackage{amssymb,amsmath,amsthm}
\usepackage{graphicx, hhline}

\theoremstyle{plain}
\newtheorem{thm}{Theorem}[section] 
\newtheorem{cor}[thm]{Corollary}
\newtheorem{prop}[thm]{Proposition}
\newtheorem{conj}[thm]{Conjecture}

\newtheorem{lem}[thm]{Lemma}
\theoremstyle{definition} 
\newtheorem{defn}[thm]{Definition}
\newtheorem{eg}[thm]{Example} 
\theoremstyle{remark}
\newtheorem{rem}[thm]{Remark}
\newtheorem{ques}[thm]{Question}

\newtheorem*{cl}{Claim}

\newtheorem*{acknowledgement}{Acknowledgments}

\def\sO{{\mathcal{O}}}
\def\J{{\mathcal{J}}} 
\def\Z{{\mathbb{Z}}}
\def\N{{\mathbb{N}}} 
\def\Q{{\mathbb{Q}}} 
\def\R{{\mathbb{R}}} 

\def\F{{\mathbb{F}}}
\def\a{{\mathfrak{a}}}
\def\m{{\mathfrak{m}}}
\def\Spec{\mathop{\mathrm{Spec}}\nolimits}

\def\max{\mathop{\mathrm{max}}\nolimits}
\def\Ker{\mathop{\mathrm{Ker}}\nolimits}
\def\deg{\mathop{\mathrm{deg}}\nolimits}
\def\lct{\mathop{\mathrm{lct}}\nolimits}
\def\fpt{\mathop{\mathrm{fpt}}\nolimits}
\def\sup{\mathop{\mathrm{sup}}\nolimits}

\def\hsymbl#1{\smash{\hbox{\LARGE$#1$}}}
\def\hsymbu#1{\smash{\lower1.4ex\hbox{\LARGE$#1$}}}

\title{Log canonical thresholds of binomial ideals}

\author{Takafumi Shibuta}
\address{Department of Mathematics, Rikkyo University/JST CREST, Nishi-Ikebukuro, Tokyo 171-8501, Japan}
\email{shibuta@rikkyo.ac.jp}

\author{Shunsuke Takagi}
\address{Department of Mathematics, Kyushu University, 6-10-1, Hakozaki, Higashi-ku, Fukuoka 812-8581, Japan}
\email{stakagi@math.kyushu-u.ac.jp}


\dedicatory{Dedicated to Professor~Toshiyuki~Katsura on the~occasion of his~sixtieth~birthday.}

\begin{document}

\begin{abstract}
We prove that the log canonical thresholds of a large class of binomial ideals, such as complete intersection binomial ideals and the defining ideals of space monomial curves, are computable by linear programming.   
\end{abstract}

\maketitle

\section*{Introduction}
The log canonical threshold is an invariant of singularities which plays an important role in higher-dimensional algebraic geometry. 
Let $\a \subseteq (x_1, \dots, x_n)$ be an ideal of the polynomial ring $k[x_1, \dots, x_n]$ over a field $k$ of characteristic zero. 
Since the log canonical threshold $\lct_0(\a)$ of $\a$ at the origin is defined via a log resolution of $\a$, it is very difficult to compute it directly from the definition, and an effective method for computing log canonical thresholds is not known. 
A notable exception is the case of monomial ideals. 
Howald \cite{Ho1, Ho2} proved that $\lct_0(\a)$ is computable by linear programming when $\a$ is a monomial ideal or a principal ideal generated by a non-degenerate polynomial. 
In this paper, we initiate the study of log canonical thresholds of binomial ideals. We then prove that the log canonical thresholds of a large class of binomial ideals, such as complete intersection binomial ideals  and the defining ideals of space monomial curves, are still computable by linear programming. Our main result is stated as follows: 

\begin{thm}[\textup{Theorems \ref{CImainresult} and \ref{nonCImonomialcurve}}]\label{main theorem}
Let $k$ be a field of characteristic zero and $\a=(f_1, \dots, f_r) \subseteq (x_1, \dots, x_n)$ be an ideal of $k[x_1, \dots, x_n]$ generated by binomials $f_i=x_1^{a_{i1}} \cdots x_n^{a_{in}}-\gamma_i x_1^{b_{i1}} \cdots x_n^{b_{in}}$,  
where 
$a_{ij}, b_{ij} \in \Z_{\ge 0}$ and $\gamma_i \in k$ for all $i=1, \dots, r$ and $j=1, \dots, n$.  
Suppose that $\a$ contains no monomials and, in addition, that one of the following conditions is satisfied:
\begin{enumerate}
\item $f_1, \dots, f_r$ form a regular sequence for $k[x_1, \dots, x_n]$,
\item $f_1, \dots, f_r$ form the canonical system of generators of the defining ideal of a monomial curve in $\mathbb{A}^3_k$ $($in this case, $r \le 3)$. 
\end{enumerate}
Then the log canonical threshold $\lct_0(\a)$ of $\a$ at the origin is equal to 
\begin{multline*}
\max\biggl\{\sum_{i=1}^r (\mu_i+\nu_i) \bigg| \\
\sum_{i=1}^r (a_{ij} \mu_i+b_{ij} \nu_i) \leq 1 \textup{ for all $1 \leq j \leq n$},  \; \mu_i+\nu_i \le 1,  \ \mu_i, \nu_i, \in \Q_{\ge 0} \biggr\}.
\end{multline*}
\end{thm}

The proof depends on two techniques. 
The first technique is the summation formula of multiplier ideals \cite[Theorem 3.2]{Ta}, which tells us that $\lct_0(\a)$ is equal to 
\begin{equation}\tag{$\star$}
\sup\{\lambda_1+\dots+\lambda_r \mid \J(f_1^{\lambda_1} \cdots f_r^{\lambda_r})_0=k[x_1, \dots, x_n]_{(x_1, \dots, x_n)}, \; 0 \le 
\lambda_i <1 \},
\end{equation}
where $\J(f_1^{\lambda_1} \cdots f_r^{\lambda_r})$ is the multiplier ideal associated to $f_1^{\lambda_1} \cdots f_r^{\lambda_r}$ (see Definition \ref{multdef} for the definition of multiplier ideals). 
Let $\a_{f_i}$ be the ideal generated by monomials appearing in $f_i$ for each $1 \leq i \leq r$. Since $\a_{f_i}$ contains $f_i$, $(\star)$ is less than or equal to 
\begin{equation}\tag{$\star\star$}
\sup\{\lambda_1+\dots+\lambda_r \mid \J(\a_{f_1}^{\lambda_1} \cdots \a_{f_r}^{\lambda_r})_0=k[x_1, \dots, x_n]_{(x_1, \dots, x_n)}, \; 0 \le \lambda_i < 1 \}.
\end{equation}
It then follows from Howald's result that $(\star\star)$ coincides with the optimal value of the linear programming problem stated in Theorem \ref{main theorem}, and consequently we obtain one inequality in Theorem \ref{main theorem}. 

The second technique is reduction from characteristic zero to positive characteristic. 
For simplicity, we assume that $\a$ is an ideal of $\Q[x_1, \dots, x_n]$ and denote by $\a_p \subseteq \F_p[x_1, \dots, x_n]_{(x_1, \dots, x_n)}$ its reduction to characteristic $p$, where $p$ is a sufficiently large prime number. 
Then the F-pure threshold $\fpt(\a_p)$ is defined to be $\displaystyle{\lim_{e \to \infty}\frac{\nu_{\a_p}(p^e)}{p^e}}$, where $\nu_{\a_p}(p^e):=\max\{r \in \Z_{\ge 0} \mid \a_p^r \not\subseteq (x_1^{p^e}, \dots, x_n^{p^e})\}$. 
It follows from a result of Hara and Yoshida \cite{HY} that the limit $\lim_{p \to \infty} \fpt(\a_p)$ of F-pure thresholds coincides with the log canonical threshold $\lct_0(\a)$ of $\a$ at the origin. 
Therefore, in order to estimate $\lct_0(\a)$, it is enough to estimate $\fpt(\a_p)$ for infinitely many $p$. 
Under the assumption of Theorem \ref{main theorem}, we show that $\fpt(\a_p)$ is greater than or equal to the optimal value of the linear programming problem in Theorem \ref{main theorem} whenever $p \equiv 1 \textup{ mod } N$, where $N$ is a fixed positive integer. As a result, we obtain the reverse inequality in Theorem \ref{main theorem}. 

In the process of proving Theorem \ref{main theorem}, we give an affirmative answer to the conjecture \cite[Conjecture 3.6]{MTW} (see also \cite[Problem 3.7]{MTW}) due to Musta\c{t}\v{a}, Watanabe and second author, when $\a$ is a complete intersection binomial ideal or the defining ideal of a space monomial curve.

\section{Preliminaries}
\subsection{Log canonical thresholds}
In this subsection, we recall the definitions of multiplier ideals and log canonical thresholds. Our main reference is \cite{L}. 

Let $X$ be a nonsingular algebraic variety over a field $k$ of characteristic zero and $\a \subseteq \sO_X$ be an ideal sheaf of $X$. 
A \textit{log resolution} of $(X,\a)$ is a proper birational morphism $\pi:\widetilde{X} \to X$ with $\widetilde{X}$ a nonsingular variety such that $\a \sO_{\widetilde{X}}=\sO_{\widetilde{X}}(-F)$ is an invertible sheaf and that $\mathrm{Exc}(\pi) \cup \mathrm{Supp}(F)$ is a simple normal crossing divisor. 

\begin{defn}\label{multdef}
In the above situation, let $t>0$ be a real number. 
Fix a log resolution $\pi:\widetilde{X} \to X$ with $\a \sO_{\widetilde{X}}=\sO_{\widetilde{X}}(-F)$. 
The \textit{multiplier ideal} $\J(\a^t)$ of $\a$ with exponent $t$ is 
\[
\J(\a^t)=\mathcal{J}(X,\a^t)=\pi_*\sO_{\widetilde{X}}(K_{\widetilde{X}/X} - \lfloor tF \rfloor),
\]
where $K_{\widetilde{X}/X}$ is the relative canonical divisor of $\pi$. 
This definition is independent of the choice of the log resolution $\pi$. 
\end{defn}

\begin{defn}
In the above situation, fix a point $x \in X$ lying in the zero locus of $\a$. 
The \textit{log canonical threshold} of $\a$ at $x \in X$ is  
\[
\lct_x(\a)=\sup\{t \in \R_+ \mid \mathcal{J}(\a^t)_x=\sO_{X,x}\} 
\]
(when $x$ is not contained in the zero locus of $\a$, we put $\lct_x(\a)=\infty$). 
The log canonical threshold $\lct_x(\a)$ is a rational number. 
\end{defn}

When the ideal $\a$ is a monomial ideal or a principal ideal generated by a non-degenerate polynomial, there exists a combinatorial description of the multiplier ideal $\J(\a^t)$ by Howald \cite{Ho1}, \cite{Ho2}. 

\begin{prop}[\cite{Ho1}, \cite{Ho2}]\label{Howaldresults}
Let $k$ be a field $($of characteristic zero$)$. 

$(1)$ Let $\a$ be a monomial ideal of $k[x_1, \dots, x_n]$ and $P(\a) \subseteq \R^d$ be the Newton polytope of $\a$. 
Then for every real number $t>0$, 
$$\J(\a^t)=(x^{\mathbf{c}} \mid \mathbf{c} + \mathbf{1} \in \mathrm{Int}(t \cdot P(\a)) \cap \N^n),$$
where $\mathbf{1}:=(1, \dots, 1) \in \N^n$. 
In particular, if $\a=(x^{\mathbf{c}_1}, \dots, x^{\mathbf{c}_s})$, then 
\begin{align*}
\lct_0(\a)&=\sup\{t \in \R_+ \mid \mathbf{1} \in t \cdot P(\a) \}\\
                &=\max\left\{ \sum_{j=1}^s \lambda_j \Bigg| \sum_{j=1}^s \mathbf{c}_j \lambda_j \le \mathbf{1}, \ \lambda_j \in \Q_{\ge 0}\right\}.
\end{align*}

$(2)$ Let $f \in (x_1, \dots, x_n)$ be a non-degenerate polynomial of $k[x_1, \dots, x_n]$ $($see \cite{Ho2} for the definition of non-degenerate polynomials. For example, every binomial  is non-degenerate$)$. 
Let $\a_f \subseteq k[x_1 \dots, x_n]$ denote the term ideal of $f$, that is, the ideal generated by the monomials appearing in $f$. Then for every real number $t>0$, 
$$\J(f^t)=f^{\lfloor t \rfloor}\J(\a_f^{t-\lfloor t \rfloor}).$$
In particular, if $f=\sum_{j=1}^s \gamma_j x^{\mathbf{c}_j}$ where $\gamma_j \in k^*$ for all $j=1, \dots, s$, then
$$\lct_0(f)=\lct_0(\a_f)=\max\left\{\sum_{j=1}^s \lambda_j \Bigg| \sum_{j=1}^s \mathbf{c}_j \lambda_j \le \mathbf{1}, \ \lambda_j \in \Q_{\ge 0}\right\}.$$
\end{prop}

Since the multiplier ideal $\J(\a^t)$ is defined via a log resolution of $\a$, it is difficult to compute the log canonical threshold $\lct_x(\a)$ in general, even when the ideal $\a$ is generated by binomials. 
\begin{eg}
Let   
$\a=(x^3-yz, y^2-xz, z^2-x^2y) \subseteq k[x,y,z]$ be the defining ideal of the monomial curve $\Spec k[t^3,t^4,t^5]$ in the affine space $\mathbb{A}_k^3$, where $k$ is a field. 
We consider the following sequence of blowing-ups:
\[
\mathbb{A}_k^3=X\stackrel{f_1}{\longleftarrow}X_1\stackrel{f_2}{\longleftarrow}X_2\stackrel{f_3}{\longleftarrow}X_3\stackrel{f_4}{\longleftarrow}X_4\stackrel{f_5}{\longleftarrow}X_5=\widetilde{X}.
\]
We denote by $C_i$ the strict transform of $C=V(\a)$ on $X_i$ and by $E_i$ the exceptional divisor of $f_i$ (and we use the same letter for its strict transform). 
Let  
$f_1$ be the blowing-up at the origin, 
$f_2$ be the blowing-up at the point $(C_1\cap E_1)_{\mathrm{red}}$, 
$f_3$ be the blowing-up at the point $C_2\cap E_2$, 
$f_4$ be the blowing-up at the point $C_3\cap E_3$ 
and $f_5$ be the blowing-up along the curve $C_4$. Then 
$\pi:=f_1\circ \dots\circ f_5:\widetilde{X}\to X$ is a log resolution of $\a$, 
and we have

\begin{align*}
K_{\widetilde{X}/X}&=2E_1+4E_2+8E_3+12E_4+E_5,\\
\a\mathcal{O}_{\widetilde{X}}&=\mathcal{O}_{\widetilde{X}}(-2E_1-3E_2-6E_3-9E_4-E_5).
\end{align*}
 Thus,
 $$\lct_0(\a)=\min\left\{\frac{2+1}{2}, \frac{4+1}{3}, \frac{8+1}{6}, \frac{12+1}{9}, \frac{1+1}{1}\right\}=\frac{13}{9}.$$
Even in this case, it is not so easy to determine all jumping coefficients of $\a$. The reader is referred to \cite{Sh} for the computation of further jumping coefficients of $\a$. 
\end{eg}

\subsection{F-pure thresholds}
In this subsection, we recall the definitions of generalized test ideals introduced by Hara and Yoshida in \cite{HY} and of F-pure thresholds introduced by Watanabe and the second author in \cite{TW}. 

Let $R$ be a Noetherian ring containing a field of characteristic $p>0$. 
The ring $R$ is called \textit{F-finite} if $R$ is a finitely generated module over its subring $R^p = (a^p \in R \mid a \in R)$. 
For each $e \in \N$, if $J$ is an ideal in $R$, then $J^{[p^e]}$ denotes the ideal $(x^{p^e} \mid x \in J)$.

Since we restrict ourselves to the case of an ambient nonsingular variety in this paper, we refer to Blickle--Musta\c{t}\v{a}-Smith's characterization \cite{BMS} as the definition of generalized test ideals.  

\begin{defn}[\textup{\cite[Definition 2.9, Proposition 2.22]{BMS}}]\label{testidealdef}
Let $R$ be an F-finite regular ring of prime characteristic $p$ and $\a$ be an ideal of $R$. 
For a given real number $t>0$, the generalized test ideal $\tau (\a^t)$ of $\a$ with exponent $t$ is the unique smallest ideal $J$ with respect to inclusion, such that
\[
\a^{\lceil qt \rceil}\subseteq J^{[q]},
\]
for all sufficiently large $q=p^e$. 
\end{defn}

\begin{defn}[\textup{\cite[Definition 2.1]{TW}}]
Let the notation be the same as in Definition \ref{testidealdef}.  
The F-pure threshold of $(R,\a)$ is 
\[
\fpt(\a)=\sup\{t \in \R_+ \mid \tau(\a^t)=R\}.
\]
If $(R,\m)$ is a regular local ring, then for each $e \in \N$, we set $\nu_{\a}(p^e)$ to be the largest nonnegative integer $r$ such that $\a^r\not\subseteq \m^{[p^e]}$.  
Then
$$\fpt(\a)=\lim_{e \to \infty}\frac{\nu_{\a}(p^e)}{p^e}.$$ 
\end{defn}

Now we briefly review the correspondence between multiplier ideals and generalized test ideals. 

Let $A$ be the localization of $\Z$ at some nonzero integer $a$. 
We fix a nonzero ideal $\a$ of the polynomial ring $A[x_1, \dots, x_n]$ such that $\a \subseteq (x_1, \dots, x_n)$. 
Let $\a_{\Q}:=\a \cdot \Q[x_1, \dots, x_n]$ and $\a_p := \a \cdot \F_p[x_1, \dots, x_n]_{(x_1, \dots, x_n)}$, where $p$ is a prime number which does not divide $a$ and $\F_p:= \Z/p\Z$.  
We call the pair $(\F_p[x_1, \dots, x_n]_{(x_1, \dots, x_n)},\a_p)$ the reduction of $(\Q[x_1, \dots, x_n],\a_\Q)$ to characteristic $p$. 
Let $\pi_\Q: Y_\Q\to \mathbb{A}_{\Q}^n$ be a log resolution of $\a_\Q$ (the existence of such a morphism is guaranteed by Hironaka's desingularization theorem \cite{Hi}).  
After further localizing $A$, we may assume that $\pi_{\Q}$ is obtained by extending the scalars from a morphism $\pi:Y \to \mathbb{A}_A^n$. 
For sufficiently large $p\gg 0$, the morphsim $\pi$ induces a log resolution $\pi_p=Y_p\to \Spec \F_p[x_1, \dots, x_n]_{(x_1, \dots, x_n)}$ of $\a_p$, and we can use $\pi_p$ to define the multiplier ideal $\J(\a_p^t)$ for a given real number $t>0$. 
Then $\mathcal{J}(\a_p^t)$ is the reduction of the multiplier ideal $\mathcal{J}(\a_\Q^t)$ of $\a_{\Q}$ to characteristic $p$. 

Hara and Yoshida discovered a connection between $\mathcal{J}(\a_p^t)$ and $\tau(\a_p^t)$ in \cite{HY}. 

\begin{thm}[\textup{\cite[Theorem 3.4, Proposition 3.8]{HY}}]
Let the notation be as above. 
\begin{enumerate}
\item If $p$ is sufficiently large, then for every real number $t>0$, we have 
$$\tau(\a_p^t) \subseteq \J(\a_p^t).$$

\item
For a given real number $t>0$, if $p$ is sufficiently large $($how large $p$ has to be depends on $t)$, then  
$$\tau(\a_p^t)=\mathcal{J}(\a_p^t).$$ 
\end{enumerate}
\end{thm}

We reformulate the above results in terms of thresholds.

\begin{cor}\label{lct vs fpt}
Let the notation be as above. 
\begin{enumerate}
\item
If $p \gg 0$, then $\fpt(\a_p) \leq \lct_0(\a_\Q)$.

\item
$\lct_0(\a_{\Q})=\lim_{p\to \infty}\fpt(\a_p)$. 
\end{enumerate}
\end{cor}

In particular, if there exist $M \in \Q$ and $N \in \N$ such that $M(q-1)=\nu_{\a_p}(q)$ for all $q=p^e$ whenever $p \equiv 1 \textup{ mod } N$, then one has $\lct_0(\a_{\Q})=M$. 

\begin{conj}[\textup{\cite[Conjecture 3.6]{MTW}}]\label{infinity primes}
In the above situation, there are infinitely many primes $p$ such that $\fpt(\a_p)=\lct(\a_{\Q})$.
\end{conj}

Thanks to Corollary \ref{lct vs fpt}, we can compute log canonical thresholds using F-pure thresholds. 
We give an easy example here. 
\begin{eg}
Let $f=x^a+y^b \in \Z[x,y]$ with $a, b \geq 2$ integers, and we will compute $\lct_0(f_{\Q})$ using $\nu_{f_p}(q)$. 
Choose any prime number $p$ such that $p \equiv 1 \textup{ mod } ab$.
Since the binomial coefficient $\binom{(1/a+1/b)(q-1)}{(1/a)(q-1)}$ is nonzero in $\F_p$ for all $q=p^e$ by Lemma \ref{Lucas},  the term $(xy)^{q-1}$ appears in the expansion of $f_p^{(1/a+1/b)(q-1)}$. This implies that $(1/a+1/b)(q-1) \leq \nu_{f_p}(q)$ for all $q=p^e$, and its reverse inequality is easy to check. Thus, by virtue of Corollary \ref{lct vs fpt}, one has $\lct_0(f_{\Q})=1/a+1/b$. 
\end{eg}

In the above example, we used the following lemma, which we will also need later. 

\begin{lem}[Lucas]\label{Lucas}
Let $p$ be a prime number, and let $m$ and $n$ be integers with $p$-adic expansions $m=\sum m_ip^i$ and $n=\sum n_ip^i$. 
Then 
\[
\binom {m}{n} = \prod_i \binom {m_i}{n_i} \mbox{~~~in~~~} \F_p. 
\]
In particular, if $0<r_1\le r_2\le 1$ are rational numbers such that $r_1(p-1)$ and $r_2(p-1)$ are integers, then for all $e \in \N$, we have 
\[
\binom{r_1(p^e-1)}{r_2(p^e-1)}=\binom{r_1(p-1)}{r_2(p-1)}^e\neq 0 \mbox{~~~in~~~} \F_p~.
\]
\end{lem}

\section{Complete intersection case}
In this section, we will prove that the log canonical thresholds of complete intersection binomial ideals are computable by linear programming. We start with our main technical result. 
\begin{prop}\label{criterion}
Let $S:=k[x_1, \dots, x_n]$ be the $n$-dimensional polynomial ring over a field $k$ of characteristic zero. 
Let $\a=(f_1, \dots, f_r)$ be an ideal of $S$ generated by binomials $f_i=x^{\mathbf{a}_i}-\gamma_ix^{\mathbf{b}_i}$, where $\mathbf{a}_i=(a_{i1}, \dots, a_{in}), \mathbf{b}_i=(b_{i1}, \dots, b_{in}) \in \Z_{\ge 0}^n \setminus \{\mathbf{0}\}$ and $\gamma_i \in k^*$  for all $i=1, \dots, r$. 
Put
$$
A:=
\begin{pmatrix}
a_{11} & \ldots & a_{r1} & b_{11} & \ldots & b_{r1} \\
\vdots & \ddots& \vdots& \vdots& \ddots& \vdots \\
a_{1n} & \ldots & a_{rn} & b_{1n} & \ldots & b_{rn} \\
1 &  &  \hsymbu{0} & 1&  &  \hsymbu{0}  \\
& \ddots&  & & \ddots & \\
\hsymbl{0} & & 1 & \hsymbl{0} &  & 1
\end{pmatrix} \in M_{\Z}(n+r, 2r),
$$
and consider the following linear programming problem:
\begin{align*}
&\textup{Maximize: } \sum_{i=1}^r (\mu_i+\nu_i) \\
&\textup{Subject to: } A \ (\mu_1, \dots, \mu_r, \nu_1, \dots, \nu_r)^{\mathrm{T}} \le \mathbf{1}, \, \mu_i, \nu_i \in \Q_{\ge 0}.
\end{align*}
Suppose that there exists an optimal solution $(\mathbf{\mu}, \mathbf{\nu})$ such that $A \ (\mathbf{\mu}, \mathbf{\nu})^{\mathrm{T}} \neq A \ (\mathbf{\mu'}, \mathbf{\nu'})^{\mathrm{T}}$ for all other optimal solutions $(\mathbf{\mu'}, \mathbf{\nu'}) \neq (\mathbf{\mu}, \mathbf{\nu})$. Then the following holds.
\begin{enumerate}
\item 
The log canonical threshold $\lct_0(\a)$ is equal to the optimal value $\sum_{i=1}^r (\mu_i+\nu_i)$. 

\item
When the $\gamma_i$ are rational numbers, put  
$\a_p:=(f_1, \dots, f_r) \cdot \F_p [x_1, \dots, x_n]_{(x_1, \dots, x_n)}$
for sufficiently large $p \gg 0$.  
Then there exists an integer $N \geq 1$ such that $\lct_0(\a)=\fpt(\a_p)$ whenever $p \equiv 1 \textup{ mod } N$. 
\end{enumerate}
\end{prop}

\begin{proof}
By virtue of the summation formula of multiplier ideals (see \cite[Theorem 3.2]{Ta}), one has 
$$\J(\a^t)=\sum_{\genfrac{}{}{0pt}{2}{t=\lambda_1+\dots+\lambda_r}{\lambda_1, \dots, \lambda_r \ge 0}}f_1^{\lfloor \lambda_1\rfloor} \cdots f_r^{\lfloor \lambda_r \rfloor}\J(f_1^{\lambda_1-\lfloor \lambda_1 \rfloor} \cdots f_r^{\lambda_r-\lfloor \lambda_r \rfloor})$$
for all real numbers $t>0$. 
Let $\a_{f_i}$ be the term ideal of $f_i$ for each $i=1, \dots, r$. 
Since $\a_{f_i}$ contains $f_i$, 
\begin{align*}
\lct_0(\a)&=\sup\{\lambda_1+\dots+\lambda_r \mid \J(f_1^{\lambda_1} \cdots f_r^{\lambda_r})_0=\sO_{X,0}, \; 0 \le \lambda_i <1 \}\\
&\leq \sup\{\lambda_1+\dots+\lambda_r \mid \J(\a_{f_1}^{\lambda_1} \cdots \a_{f_r}^{\lambda_r})_0=\sO_{X,0}, \; 0 \le \lambda_i < 1 \}.
\end{align*}
Applying Proposition \ref{Howaldresults} (i), one can see that the last term in the above inequality coincides with 
\begin{align*}
&\max\left\{\sum_{i=1}^r (\mu_i+\nu_i) \Bigg| \sum_{i=1}^r (\mathbf{a}_i \mu_i+\mathbf{b}_i \nu_i) \le \mathbf{1}, \; \mu_i+\nu_i \le 1, \ \mu_i, \nu_i \in \Q_{\ge 0} \right\}\\
=&\max\left\{\sum_{i=1}^r (\mu_i+\nu_i) \Bigg| A \ (\mu_1, \dots, \mu_r, \nu_1, \dots, \nu_r)^{\mathrm{T}} \le \mathbf{1}, \ \mu_i, \nu_i \in \Q_{\ge 0} \right\}.
\end{align*}
Consequently, we obtain one inequality in the theorem. 

Next, we prove the converse inequality. Fix an optimal solution 
$$(\mathbf{\mu}, \mathbf{\nu})=(\mu_1, \dots, \mu_r, \nu_1, \dots, \nu_r)$$ 
such that $A \ (\mathbf{\mu}, \mathbf{\nu})^{\mathrm{T}} \neq A \ (\mathbf{\mu'}, \mathbf{\nu'})^{\mathrm{T}}$ for all other optimal solutions $(\mathbf{\mu'}, \mathbf{\nu'}) \neq (\mathbf{\mu}, \mathbf{\nu})$.
We then prove that $\sum_{i=1}^r (\mu_i+\nu_i) \le \lct_0(\a)$ making use of F-pure thresholds. 

For simplicity, we may assume that $\gamma_i$ is a rational number for all $i=1, \dots, r$, and let $\a_p:=(f_1, \dots, f_r) \cdot \F_p[x_1, \dots, x_n]_{(x_1, \dots, x_n)}$ for sufficiently large $p \gg 0$ (even if $\gamma_i \notin \Q$, we can still consider the reduction of $\a$ to characteristic $p \gg 0$).  
We take the integer $N \geq 1$ to be the least common multiple of the denominators of the $\mu_i, \nu_i$, so that $\mu_i(p-1), \nu_i(p-1)$ are integers for all $i=1, \dots, r$ whenever $p \equiv 1 \textup{ mod } N$. 
By virtue of Corollary \ref{lct vs fpt}, it is enough to show that for such prime numbers $p \gg 0$,  $\sum_{i=1}^r(\mu_i+\nu_i)(q-1)\leq \nu_{\a_p}(q)$ for all $q=p^e$. 
Therefore, from now on, we consider only such $p$. 

Let $m_1, \dots, m_n$ be nonnegative integers such that 
$$
A 
\left(
\begin{array}{c}
\mu_1(q-1)\\
\vdots \\
\mu_r(q-1)\\
\nu_1(q-1)\\
\vdots \\
\nu_r(q-1)
\end{array}
\right)
=
\left(
\begin{array}{c}
m_1\\
\vdots \\
m_n\\
(\mu_1+\nu_1)(q-1)\\
\vdots\\
(\mu_r+\nu_r)(q-1)
\end{array}
\right).
$$
Then $m_i \le q-1$ for all $i=1, \dots, r$. 
The coefficient of the term $x_1^{m_1} \cdots x_n^{m_n}$ in the expansion of $f_1^{(\mu_1+\nu_1)(q-1)} \cdots f_r^{(\mu_r+\nu_r)(q-1)}$ is

\begin{equation} \tag*{($\star$)}
\prod_{i=1}^r  \sum_{s_i, t_i} (-\gamma_i)^{t_i}\binom{(\mu_i+\nu_i)(q-1)}{s_i},
\end{equation}
where the summation runs over all $(s_1, \dots, s_r, t_1, \dots, t_r) \in \Z_{\geq 0}^{2r}$ such that 
$$
A 
\left(
\begin{array}{c}
s_1\\
\vdots \\
s_r\\
t_1\\
\vdots \\
t_r
\end{array}
\right)
=
\left(
\begin{array}{c}
m_1\\
\vdots \\
m_n\\
(\mu_1+\nu_1)(q-1)\\
\vdots\\
(\mu_r+\nu_r)(q-1)
\end{array}
\right).
$$
Note that $(\frac{s_1}{q-1}, \dots, \frac{s_r}{q-1}, \frac{t_1}{q-1}, \dots, \frac{t_r}{q-1})$ is an optimal solution of the linear programming problem stated in the proposition. 
Thus, by the definition of the optimal solution $(\mu_1, \dots, \mu_r, \nu_1, \dots, \nu_r)$, 
the coefficient $(\star)$ is equal to 
$$\prod_{i=1}^r \left(-\gamma_i\right)^{\nu_i(q-1)}\binom{(\mu_i+\nu_i)(q-1)}{\mu_i(q-1)}.
$$
It follows from Lemma \ref{Lucas} that this coefficient is nonzero in $\F_p$, 
which means that the term $x_1^{m_1} \cdots x_n^{m_n}$ appears in the expansion of $f_1^{(\mu_1+\nu_1)(q-1)} \cdots f_r^{(\mu_r+\nu_r)(q-1)}$ in  $\F_p[x_1, \dots, x_n]_{(x_1, \dots, x_n)}$. 
Since $m_i \leq q-1$ for all $i=1, \dots, r$, one has 
$$f_1^{(\mu_1+\nu_1)(q-1)} \cdots f_r^{(\mu_r+\nu_r)(q-1)} \notin (x_1^q, \dots, x_n^q) \textup{ in $\F_p[x_1, \dots, x_n]_{(x_1, \dots, x_n)}$}$$
for all $q=p^e$. 
That is, $\sum_{i=1}^r(\mu_i+\nu_i)(q-1) \leq \nu_{\a_p}(q)$ for all $q=p^e$. 
\end{proof}

\begin{ques}
If $\a$ contains no monomials and $f_1, \dots, f_r$ are a system of minimal binomial generators for $\a$, then is the assumption of Proposition \ref{criterion} satisfied?
We will see later that the answer is ``yes" if $f_1, \dots, f_r$ form a regular sequence (Theorem \ref{CImainresult} or define a space monomial curve (Theorem \ref{nonCImonomialcurve}). 
\end{ques}

\begin{rem}
Since polynomial-time algorithms for linear programming are known to exist (however, the most practical  algorithm, the simplicial method, is exponential in time), we can compute log canonical thresholds of binomial ideals in polynomial-time if the assumption of Proposition \ref{criterion} is satisfied. 
\end{rem}

We use Proposition \ref{criterion} to generalize Howald's result \cite[Example 5]{Ho1} (see also Proposition \ref{Howaldresults} (1)). 
\begin{thm}\label{CImainresult}
Let $k$ be a field of characteristic zero and $\a=(f_1, \dots, f_r, g_1, \dots, g_s)$ be an ideal of $k[x_1, \dots, x_n]$ generated by binomials $f_i=x^{\mathbf{a}_i}-\gamma_ix^{\mathbf{b}_i}$ and monomials $g_j=x^{\mathbf{c}_j}$, where $\mathbf{a}_i, \mathbf{b}_i, \mathbf{c}_j \in \Z_{\ge 0}^n \setminus \{\mathbf{0}\}$ and $\gamma_i \in k^*$ for all $i=1, \dots, r$ and $j=1, \dots, s$. 
We assume that the ideal $(f_1, \dots, f_r)$ contains no monomials and that $f_1, \dots, f_r$ form a regular sequence for $k[x_1, \dots, x_n]$. 
\begin{enumerate}
\item
The log canonical threshold $\lct_0(\a)$ of $\a$ at the origin is equal to
\begin{multline*}
\hspace*{0.9cm}\max\biggl\{\sum_{i=1}^r (\mu_i+\nu_i)+\sum_{j=1}^s \lambda_j \bigg| \\
\sum_{i=1}^r (\mathbf{a}_i \mu_i+\mathbf{b}_i \nu_i)+\sum_{j=1}^s \mathbf{c}_j \lambda_j \le \mathbf{1}, \; \mu_i+\nu_i \le 1,  \ \mu_i, \nu_i, \lambda_j \in \Q_{\ge 0} \biggr\}.
\end{multline*}
\item
When the $\gamma_i$ are rational numbers, we denote 
$$\a_p:=(f_1, \dots, f_r, g_1, \dots, g_s) \cdot \F_p [x_1, \dots, x_n]_{(x_1, \dots, x_n)}$$
for sufficiently large $p \gg 0$. 
Then there exists an integer $N \geq 1$ such that $\lct_0(\a)=\fpt(\a_p)$ whenever $p \equiv 1 \textup{ mod } N$.
\
\end{enumerate}
\end{thm}

\begin{proof}
Since the log canonical threshold $\lct_0(\a)$ does not change after an extension of the base field $k$ (see \cite[Proposition 2.9]{dFM}), we may assume that $k$ is algebraically closed. 
Since the ideal $(f_1, \dots, f_r)$ does not contain any monomial, there exist $\delta_1, \dots, \delta_n \in k^*$ such that $(f_1, \dots, f_r) \subseteq (x_1-\delta_1, \dots, x_n-\delta_n)$. 
After a suitable coordinate change (that is, $x_l \mapsto \delta_l x_l$ for each $l=1, \dots, n$), we can assume that $(f_1, \dots, f_r)$ is contained in $(x_1-1, \dots, x_n-1)$, which is equivalent to saying that $\gamma_i=1$ for all $i=1, \dots, r$. 

First we consider the case where $\a=(f_1, \dots, f_r)$. 
We take the $(n+r) \times 2r$ matrix 
$$
A:=
\begin{pmatrix}
a_{11} & \ldots & a_{r1} & b_{11} & \ldots & b_{r1} \\
\vdots & \ddots& \vdots& \vdots& \ddots& \vdots \\
a_{1n} & \ldots & a_{rn} & b_{1n} & \ldots & b_{rn} \\
1 &  &  \hsymbu{0} & 1&  &  \hsymbu{0}  \\
& \ddots&  & & \ddots & \\
\hsymbl{0} & & 1 & \hsymbl{0} &  & 1
\end{pmatrix},
$$
where $\mathbf{a}_i=(a_{i1}, \dots, a_{in})$  and  $\mathbf{b}_i=(b_{i1}, \dots, b_{in})$ for all $i=1, \dots, r$. 

\begin{cl}
$$\mathrm{rank\; }A=2r.$$
\end{cl}
\begin{proof}[Proof of Claim]
We can transform $A$ by applying sequential elementary row operations (for example, if $a_{ij} \geq b_{ij}$, then add the $(n+i)^{\rm th}$ row multiplied by $-b_{ij}$ to the $j^{\rm th}$ row) to an $(n+r) \times 2r$ matrix
$$
A':=
\begin{pmatrix}
a'_{11} & \ldots & a'_{r1} & b'_{11} & \ldots & b'_{r1} \\
\vdots & \ddots& \vdots& \vdots& \ddots& \vdots \\
a'_{1n} & \ldots & a'_{rn} & b'_{1n} & \ldots & b'_{rn} \\
1 &  &  \hsymbu{0} & 1&  &  \hsymbu{0}  \\
& \ddots&  & & \ddots & \\
\hsymbl{0} & & 1 & \hsymbl{0} &  & 1
\end{pmatrix},
$$
where $a'_{ij}, b'_{ij} \in \Z_{\ge 0}$ such that $a'_{ij} b'_{ij}=0$ for each $i=1,\dots, r$ and $j=1, \dots, n$. 
Let $\a'$ be the binomial ideal associated to $A'$, that is, $\a'=(f_1', \dots, f_r')$ is generated by binomials $f_i':=x^{\mathbf{a}'_i}- x^{\mathbf{b}'_i}$, where  $\mathbf{a}'_i=(a'_{i1}, \dots, a'_{in})$  and  $\mathbf{b}'_i=(b'_{i1}, \dots, b'_{in})$ for all $i=1, \dots, r$. 
Let $S_x=k[x_1^{\pm}, \dots, x_n^{\pm}]$ be the Laurent polynomial ring. 
Note that $\a' S_x=\a S_x$ because $f_i/f_i'$ is a monomial in $S$ for all $i=1, \dots, r$. 
Then, by  \cite[Lemma 4.39]{R} (see also \cite[Theorem 2.1]{ES}), one has 
$$\mathrm{ht\;}\a S_x+r=\mathrm{ht\;}\a' S_x+r \leq \mathrm{rank\;}A'=\mathrm{rank\; }A.$$
On the other hand, since $f_1, \dots, f_r$ form a regular sequence, $r=\mathrm{ht\;} \a \le \mathrm{ht\;}\a S_x$. 
Consequently, we obtain the assertion. 
\end{proof}

By the above claim, all optimal solutions of the linear programming problem stated in the theorem satisfy the assumption of Proposition \ref{criterion}. Thus, the assertion immediately follows from Proposition \ref{criterion}.  

We now move to the general case. 
Fix any optimal solution 
$$(\mu_1, \dots, \mu_r, \nu_1, \dots, \nu_r, \lambda_1, \dots, \lambda_s)$$ 
of the linear programming problem stated in the theorem, and consider another linear programming problem:
\begin{align*}
&\textup{Maximize: } \sum_{i=1}^r (\sigma_i+\tau_i) \\
&\textup{Subject to: } A 
\left(
\begin{array}{c}
\sigma_1\\
\vdots \\
\sigma_r\\
\tau_1\\
\vdots \\
\tau_r
\end{array}
\right)
\leq 
\left(
\begin{array}{c}
1-\sum_{j=1}^s c_{j1}\lambda_j\\
\vdots \\
1-\sum_{j=1}^s c_{jn}\lambda_j\\
1\\
\vdots\\
1
\end{array}
\right), \, \, \sigma_i, \tau_i \in \Q_{\ge 0}, 
\end{align*}
where $\mathbf{c}_j=(c_{j1}, \dots, c_{jn})$ for all $j=1, \dots, s$. 
Then $(\mu_1, \dots, \mu_r, \nu_1, \dots, \nu_r)$ is obviously an optimal solution of this linear programming problem. 
Also, it follows from a similar argument to the proof of Proposition \ref{criterion} that 
if there exists an optimal solution $(\sigma, \tau)$ such that $A(\sigma, \tau)^{\rm T} \neq  A(\sigma' , \tau')^{\rm T}$ for all other optimal solutions $(\sigma' , \tau') \neq (\sigma, \tau)$, then its optimal value $\sum_{i=1}^r (\sigma_i+\tau_i)$ is equal to $\lct_0(\a)-\sum_{j=1}^s \lambda_j$. 
However, by the above claim, all optimal solutions satisfy this assumption. Thus, we have $\lct_0(\a)=\sum_{i=1}^r (\mu_i+\nu_i)+\sum_{j=1}^s \lambda_j$. 
\end{proof}

As a corollary of Theorem \ref{CImainresult}, we obtain the complete table of log canonical thresholds of complete intersection space monomial curves. 

Let $n_1, n_2, n_3 \ge 2$ be integers with greatest common divisor one. 
Let $\a \subseteq k[x,y,z]$ be the defining ideal of the complete intersection monomial curve $\Spec k[t^{n_1}, t^{n_2}, t^{n_3}]$ in $\mathbb{A}_k^3$, where $k$ is a field of characteristic zero.  
We make $k[x,y,z]$ into a graded ring as in Section $3$. 
Since $\Spec k[t^{n_1}, t^{n_2}, t^{n_3}]$ is a complete intersection in $\mathbb{A}_k^3$, after suitable permutation of the $n_i$, we may assume that  
$(n_1,n_2,n_3)=(cb_1,ca_1,a_1b_2+a_2b_1)$ for some integers $a_1, b_1, c \ge 1$ and $a_2 \ge b_2 \ge 0$ with $a_2+b_2 \ge 1$. 
Then we can write $\a=(f, g)$, where $f:=x^{a_1}-y^{b_1}$ and $g:=z^c-x^{a_2}y^{b_2}$. 

\begin{cor}
In the above situation, the following is the complete table of log canonical thresholds $\lct_0(\a)$ of complete intersection space monomial curves. 
\begin{center}
\begin{tabular}{|c|c|}
\hline
\textup{Cases} & $\lct_0(\a)$\\
\hhline{|=|=|}
$(\deg f\le\deg g)$  $\vee$ $(c=1)$ & $\frac{1}{a_1}+\frac{1}{b_1}+\frac{1}{c}$\\[1.5mm]
\hline
$(\deg f>\deg g)$ $\wedge$ $(a_2=b_2=1)$ & $(\frac{1}{a_1}+\frac{1}{b_1})\frac{1}{c}+1$\\[1.5mm]
\hline
$(\deg f>\deg g)$ $\wedge$ $(a_2=1)$ $\wedge$ $(b_2=0)$ & $\frac{1}{a_1c}+\frac{1}{b_1}+1$\\[1.5mm]
\hline
$(\deg f>\deg g)$ $\wedge$ $(c, a_2 \ge 2)$ & $\frac{1}{a_2}+(1-\frac{b_2}{a_2})\frac{1}{b_1}+\frac{1}{c}$\\[1.5mm]
\hline
\end{tabular}
\end{center}
\end{cor}

\section{Non-complete intersection case}
In this section, we compute log canonical thresholds of non-complete intersection space monomial curves. 

Let $n_1, n_2, n_3 \ge 2$ be integers with greatest common divisor one. 
Let $S:=k[x,y,z]$ be the polynomial ring over a field $k$ of characteristic zero and $R:=k[H]=k[t^{n_1},t^{n_2},t^{n_3}]$ be the numerical semigroup ring associated to $H:=\{m_1n_1+m_2n_2+m_3n_3|m_i\in \Z_{\ge 0}\}$ over $k$. 
We define the ideal $\a \subseteq S$ to be the kernel of the ring morphism $\varphi: S\to R$ sending $x$ to $t^{n_1}$, $y$ to $t^{n_2}$ and $z$ to $t^{n_3}$. 
We make $S$ into an $H$-graded ring by assigning $\deg_H x=n_1$, $\deg_H y=n_2$ and $\deg_H z=n_3$. 
Then $\a$ is a homogeneous binomial ideal.  

Suppose that $R$ is not a complete intersection. 
Then there exist integers $a_i$, $b_i, c_i \ge 1$ for $i=1,2$ such that $\a=(f_1, f_2, f_3)$, where
\[
f_1 := x^{a_1+a_2}-y^{b_1}z^{c_2}, \
f_2 := y^{b_1+b_2}-z^{c_1}x^{a_2}, \
f_3 := z^{c_1+c_2}-x^{a_1}y^{b_2}.
\]
Since $n_i$ is the length of $R/(t^{n_i})$, we have
\begin{align*}
n_1 &= (b_1+b_2)(c_1+c_2)-b_2c_1,\\
n_2 &= (c_1+c_2)(a_1+a_2)-c_2a_1,\\
n_3 &= (a_1+a_2)(b_1+b_2)-a_2b_1.
\end{align*}
Put $\alpha:=a_1/(a_1+a_2)$, $\beta:=b_1/(b_1+b_2)$ and $\gamma:=c_1/(c_1+c_2)$. 
We may assume without loss of generality that
\[
\deg_H f_1<\deg_H f_2<\deg_H f_3,  
\]
which is equivalent to saying that 
$$(1-\beta)\gamma>(1-\gamma)\alpha>(1-\alpha)\beta.$$
We remark that the degrees of the $f_i$ disagree with each other, since the substitution morphism $\varphi$ sends all monomials of the same degree to the same power of $t$.

\begin{thm}\label{nonCImonomialcurve}
In the above situation, $\a=(f_1, f_2,f_3)$ satisfies the assumption of Proposition \ref{criterion}. 
Consequently, the following holds.
\begin{enumerate}
\item
The log canonical threshold $\lct_0(\a)$ of $\a$ at the origin is equal to
$$\max\left\{\sum_{i=1}^3 (\mu_i+\nu_i) \Bigg| A \ (\mu_1, \mu_2, \mu_3, \nu_1, \nu_2, \nu_3)^{\mathrm{T}} \le \mathbf{1}, \ \mu_i, \nu_i \in \Q_{\ge 0} \right\},
$$
where 
$$A=
\left(
\begin{array}{cccccc}
a_1+a_2 & 0 & 0 & 0 & a_2 & a_1\\
0 & b_1+b_2 & 0 & b_1 & 0 & b_2\\
0 & 0 & c_1+c_2 & c_2 & c_1 & 0\\
1 & 0 & 0 & 1 & 0 & 0 \\
0 & 1 & 0 & 0 & 1 & 0\\
0 & 0 & 1 & 0 & 0 & 1\\
\end{array}
\right).
$$\\
Solving the above linear programming problem, we obtain the following table. \\

\begin{center}
\begin{tabular}{|c|c|}
\hline
\rm{Cases} & $\lct_0(\a)$\\
\hhline{|=|=|}
$b_1=c_2=1$ & $1+\frac{n_1}{n_2(1+b_2)}$\\[1.5mm]
\hline
$(b_1 \le c_2)$ $\wedge$ $(c_2 \ge 2)$& $\frac{1}{a_1+a_2}+\frac{1}{b_1+b_2}\bigl(1+\frac{b_2}{c_2}\bigr)$\\[1.5mm]
\hline
$(b_1>c_2)$ $\wedge$ $(\alpha \le \gamma)$ & $\frac{1}{a_1+a_2}+\frac{b_1+c_1}{b_1(c_1+c_2)}$\\[1.5mm]
\hline
$(b_1>c_2)$ $\wedge$ $(\alpha \ge \gamma)$ $\wedge$ $(\frac{c_1}{a_2}+\frac{c_2}{b_1} \le 1)$ & $\frac{b_1+c_1}{b_1(c_1+c_2)}+\frac{c_2}{a_2(c_1+c_2)}$\\[1.5mm]
\hline
$(b_1>c_2)$ $\wedge$ $(\alpha\ge \gamma)$ $\wedge$ $(\frac{c_1}{a_2}+\frac{c_2}{b_1} > 1)$ & $\frac{1}{a_1+a_2}+\frac{1}{b_1}+\frac{a_1}{(a_1+a_2)c_1}(1-\frac{c_2}{b_1})$\\[1.5mm]
\hline
\end{tabular}
\\
\end{center}
\item
For each prime number $p$, put  $\a_p:=(f_1, f_2, f_3) \cdot \F_p [x, y, z]_{(x,y,z)}$.  
Then there exists an integer $N \geq 1$ such that $\lct_0(\a)=\fpt(\a_p)$ whenever $p \equiv 1 \textup{ mod } N$. 
\end{enumerate}
\end{thm}

\begin{proof}
We denote by $(P)$ the corresponding linear programming problem. 
Since $R$ is not a complete intersection, by an argument similar to Claim in the proof of Theorem \ref{CImainresult}, we can see that $\mathrm{rank\;}A=5$ and $\Ker A=\Q \cdot (1, 1,1,-1, -1, -1)^{\mathrm{T}}$. 
Then an optimal solution $(\mu_1, \mu_2, \mu_3, \nu_1, \nu_2, \nu_3)$ of $(P)$ satisfies the assumption of Proposition \ref{criterion} if and only if $\mu_i=\nu_j=0$ for some $1 \le i, j \le 3$. 
So, we look for a optimal solution $(\mu_1, \mu_2, \mu_3, \nu_1, \nu_2, \nu_3)$ of $(P)$ such that 
$\mu_i=\nu_j=0$ for some $1 \le i, j \le 3$. 
To do it,  the following fact is useful: if $(\mu_1, \mu_2, \mu_3, \nu_1, \nu_2, \nu_3)$ is a feasible solution of $(P)$, then 
$$(\mu_1+\nu_1) \deg_Hf_1+(\mu_2+\nu_2) \deg_Hf_2+(\mu_3+\nu_3) \deg_Hf_3 \le n_1+n_2+n_3.$$ 

\noindent\underline{In the case when $b_1=c_2=1$:} 

Let $(\mu_1, \mu_2, \mu_3, \nu_1, \nu_2, \nu_3)=\Bigl(\frac{c_1}{n_2}, \frac{c_1}{(1+b_2)n_2}, 0, 1-\frac{c_1}{n_2},  \frac{1}{n_2}, 0\Bigr)$.   
Then it is easy to see that $(\mu_1, \mu_2, \mu_3, \nu_1, \nu_2, \nu_3)$ is a feasible solution of $(P)$ and that 
$$(\mu_1+\nu_1) \deg_Hf_1+(\mu_2+\nu_2) \deg_Hf_2= n_1+n_2+n_3,$$ because 
\begin{align*}
n_1 &= (1+b_2)(c_1+1)-b_2c_1=~b_2+c_1+1,\\
n_2 &= (c_1+1)(a_1+a_2)-a_1~=(a_1+a_2)c_1+a_2,\\
n_3 &= (a_1+a_2)(1+b_2)-a_2~=(a_1+a_2)b_2+a_1. 
\end{align*}
Since $\mu_1+\nu_1=1$, we cannot add anything more to $\mu_1$ or $\nu_1$. 
Thus, since $\deg_H f_1<\deg_H f_2<\deg_H f_3$, the solution $(\mu_1, \mu_2, \mu_3, \nu_1, \nu_2, \nu_3)$ must be optimal.
By Proposition \ref{criterion}, the log canonical threshold $\lct_0(\a)$ is equal to the optimal value $1+\frac{n_1}{n_2(1+b_2)}$. 
\vspace{1em}

\noindent\underline{In the case when $b_1 \leq c_2$ and $c_2 \ge 2$:}

Let $(\mu_1, \mu_2, \mu_3, \nu_1, \nu_2, \nu_3)=\Bigl(\frac{1}{a_1+a_2}, \frac{1}{b_1+b_2}\Bigl(1-\frac{b_1}{c_2}\Bigr), 0, \frac{1}{c_2}, 0, 0\Bigr)$.
Then it is easy to check that $(\mu_1, \mu_2, \mu_3, \nu_1, \nu_2, \nu_3)$ is a feasible solution of $(P)$
and that
$$(\mu_1+\nu_1)\deg_H f_1+\mu_2\deg_H f_2=n_1+n_2+n_3, $$
because $b_1 \leq c_2$ and $c_2 \geq 2$. 
By the definition of $(P)$, we cannot add anything more to $\mu_1$ or $\nu_1$. 
Thus, since $\deg_H f_1<\deg_H f_2<\deg_H f_3$, the solution $(\mu_1, \mu_2, \mu_3, \nu_1, \nu_2, \nu_3)$ must be optimal.
By Proposition \ref{criterion}, the log canonical threshold $\lct_0(\a)$ is equal to the optimal value $\frac{1}{a_1+a_2}+\frac{1}{b_1+b_2}\Bigl(1-\frac{b_1}{c_2}\Bigr)$. 
\vspace{1em}

\noindent\underline{In the case when $b_1>c_2$:}

We consider the following linear programming problem $(Q)$:
$$\max\left\{\sum_{i=1}^6 \lambda_i \Bigg|  
B \ (\lambda_1,\dots,\lambda_6)^{\mathrm{T}} \leq \mathbf{1}, \ \lambda_i \in \Q_{\ge 0}\right\},$$
where 
$$B:=
\left(
\begin{array}{cccccc}
a_1+a_2 & 0 & 0 & 0 & a_2 & a_1\\
0 & b_1+b_2 & 0 & b_1 & 0 & b_2\\
0 & 0 & c_1+c_2 & c_2 & c_1 & 0\\
\end{array}
\right).
$$
If $(\lambda_1,\dots,\lambda_6)$ is an optimal solution of $(Q)$, then it is obvious that 
\[
	\lambda_1 = \frac{1-a_2\lambda_5 - a_1\lambda_6}{a_1+a_2},~~
	\lambda_2 = \frac{1-b_1\lambda_4 - b_2\lambda_6}{b_1+b_2},~~
	\lambda_3 = \frac{1-c_2\lambda_4 - c_1\lambda_5}{c_1+c_2}.
\]
In this case,  
\begin{align*}
\sum_{i=1}^6 \lambda_i &=\frac{1-a_2\lambda_5 - a_1\lambda_6}{a_1+a_2} +\frac{1-b_1\lambda_4 - b_2\lambda_6}{b_1+b_2}+ \frac{1-c_2\lambda_4 - c_1\lambda_5}{c_1+c_2} + \lambda_4+\lambda_5+\lambda_6 \\
&=(\gamma-\beta)\lambda_4 +(\alpha-\gamma)\lambda_5 +(\beta-\alpha)\lambda_6+\frac{1}{a_1+a_2}+\frac{1}{b_1+b_2}+\frac{1}{c_1+c_2}. 
\end{align*}
Since $(1-\beta)\gamma>(1-\gamma)\alpha>(1-\alpha)\beta$, it is easy to see that $\alpha >\beta$.  
So, the linear function $(\gamma-\beta)\lambda_4 +(\alpha-\gamma)\lambda_5 +(\beta-\alpha)\lambda_6$ achieves the maximal value when $\lambda_6=0$. 
This means that $(Q)$ is equivalent to the following linear programming problem $(Q')$ up to a constant:
$$
\max 
\left\{ (\gamma-\beta)\lambda_4 +(\alpha-\gamma)\lambda_5 
\left| 
\begin{array}{l} 
\\\\\\\\
\end{array}\right. 
\!\!\!\! \!\!\!\!
\begin{array}{r}
 a_2\lambda_5 \le 1, \ b_1\lambda_4 \le 1 \\ 
 c_2\lambda_4 +c_1\lambda_5 \le 1 \\
 \lambda_4, \lambda_5 \in\Q_{\ge 0}
	\end{array}
\right\}. 
$$
Since $(1-\beta)\gamma>(1-\gamma)\alpha$, one has $\gamma-\beta>(1-\gamma)(\alpha-\beta)$ and, in particular, $\gamma>\beta$. 

\begin{enumerate}
\item
In the case when $\alpha\le \gamma$:

\noindent$(\frac{1}{b_1}, 0)$ is an optimal solution of $(Q')$, and thus 
$\Bigl(\frac{1}{a_1+a_2}, 0, \frac{b_1-c_2 }{b_1(c_1+c_2)}, \frac{1}{b_1}, 0, 0\Bigr)$ is an optimal solution of $(Q)$. 
Since $\frac{1}{a_1+a_2}+\frac{1}{b_1} \leq 1$, it is also an optimal solution of $(P)$. 
By Proposition \ref{criterion}, the log canonical threshold $\lct_0(\a)$ is equal to the optimal value $\frac{1}{a_1+a_2}+\frac{b_1+c_1}{b_1(c_1+c_2)}$. \\

\item
In the case where $\alpha>\gamma$ and $\frac{c_1}{a_2}+\frac{c_2}{b_1}\le 1$:

\noindent$(\frac{1}{b_1},\frac{1}{a_2})$ is an optimal solution of $(Q')$, and thus
$\Bigl(0, 0, \frac{a_2b_1-a_2c_2-b_1c_1}{a_2b_1(c_1+c_2)}, \frac{1}{b_1}, \frac{1}{a_2}, 0\Bigl)$ is an optimal solution of $(Q)$. 
It is clearly an optimal solution of $(P)$, and then by Proposition \ref{criterion}, the log canonical threshold $\lct_0(\a)$ is equal to the optimal value $\frac{b_1+c_1}{b_1(c_1+c_2)}+\frac{c_2}{a_2(c_1+c_2)}$. \\

\item
In the case where $\alpha>\gamma$ and $\frac{c_1}{a_2}+\frac{c_2}{b_1} > 1$:
\begin{center}
\includegraphics*[width=10.0cm]{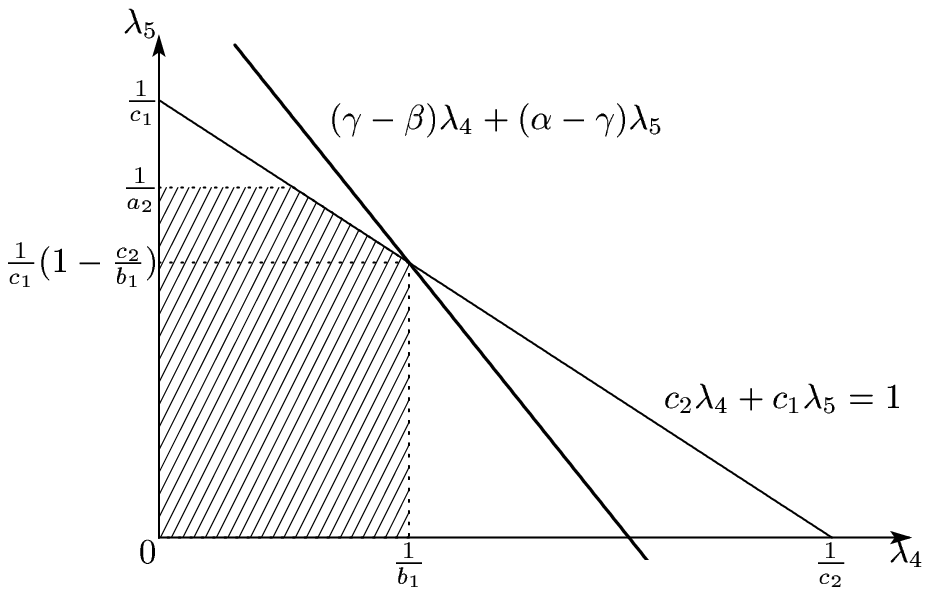}
\end{center}
\noindent First note that $c_2(\gamma-\beta)>c_1(\alpha-\gamma)$, because $(1-\beta)\gamma>(1-\gamma)\alpha$. Then $(\frac{1}{b_1},\frac{1}{c_1}\bigl(1-\frac{c_2}{b_1}\bigr))$ is an optimal solution of $(Q')$, and thus 
$$\left(\frac{b_1c_1-a_2b_1+a_2c_2}{(a_1+a_2)b_1c_1}, 0, 0, \frac{1}{b_1}, \frac{1}{c_1}\left(1-\frac{c_2}{b_1}\right), 0\right)$$ is an optimal solution of $(Q)$. 
Since $\frac{b_1c_1-a_2b_1+a_2c_2}{(a_1+a_2)b_1c_1}+\frac{1}{b_1} \le 1$,  it is also an optimal solution of $(P)$. 
By Proposition \ref{criterion}, the log canonical threshold $\lct_0(\a)$ is equal to the optimal value $\frac{1}{a_1+a_2}+\frac{1}{b_1}+\frac{a_1}{(a_1+a_2)c_1}(1-\frac{c_2}{b_1})$. 
\end{enumerate}
\end{proof}

By an argument similar to the proof of Theorem \ref{nonCImonomialcurve}, we can compute the log canonical threshold $\lct_0(\a)$ of the defining ideal $\a$ of a Gorenstein monomial curve $k[t^{n_1}, t^{n_2}, t^{n_3}, t^{n_4}]$ in $\mathbb{A}_k^4$. Here we give only one example. 

\begin{eg}
Let $\a=(x^3-zw, y^3-xz^2, z^3-y^2w, w^2-x^2y, xw-yz) \subseteq k[x,y,z,w]$ be the defining ideal of the monomial curve $k[t^8, t^{10}, t^{11}, t^{13}]$ in the affine space $\mathbb{A}_k^4$, where $k$ is a field of characteristic zero. We consider the following linear programming problem:
$$\max\left\{\sum_{i=1}^5 (\mu_i+\nu_i) \Bigg| A \ (\mu_1, \dots, \mu_5, \nu_1, \dots, \nu_5)^{\mathrm{T}} \le \mathbf{1}, \ \mu_i, \nu_i \in \Q_{\ge 0}\right\},$$
where  
$$
A:=\left(
\begin{array}{cccccccccc}
3 & 0 & 0 & 0 & 1 & 0 & 1 & 0 & 2 & 0\\
0 & 3 & 0 & 0 & 0 & 0 & 0 & 2 & 1 & 1\\
0 & 0 & 3 & 0 & 0 & 1 & 2 & 0 & 0 & 1\\
0 & 0 & 0 & 2 & 1 & 1 & 0 & 1 & 0 & 0\\
1 & 0 & 0 & 0 & 0 & 1 & 0 & 0 & 0 & 0\\
0 & 1 & 0 & 0 & 0 & 0 & 1 & 0 & 0 & 0\\
0 & 0 & 1 & 0 & 0 & 0 & 0 & 1 & 0 & 0\\
0 & 0 & 0 & 1 & 0 & 0 & 0 & 0 & 1 & 0\\
0 & 0 & 0 & 0 & 1 & 0 & 0 & 0 & 0 & 1
\end{array}
\right).
$$
Let $(\mathbf{\mu}, \mathbf{\nu}):=(\mu_1, \dots, \mu_5, \nu_1, \dots, \nu_5)$ be an optimal solution of the above linear programming problem. 
Since 
$$\Ker A=k \cdot (1,1,1,1,0, -1, -1, -1, -1, 0)^{\mathrm{T}}+k \cdot (0,1,1,0,1,0, -1, -1, 0, -1)^{\mathrm{T}},$$
if $\mu_3=\mu_4=\nu_3=\nu_4=0$, then there exists no other optimal solution $(\mathbf{\mu'}, \mathbf{\nu'}) \neq (\mathbf{\mu}, \mathbf{\nu})$ such that $A \ (\mathbf{\mu}, \mathbf{\nu})^{\mathrm{T}}=A \ (\mathbf{\mu'}, \mathbf{\nu'})^{\mathrm{T}}$. 
In this case, by Proposition \ref{criterion}, the log canonical threshold $\lct_0(\a)$ is equal to its optimal value $\sum_{i=1}^5 (\mu_i+\nu_i)$. 
Thus, we look for a optimal solution $(\mu_1, \dots, \mu_5, \nu_1, \dots, \nu_5)$ such that $\mu_3=\mu_4=\nu_3=\nu_4=0$. It is easy to check that $(\frac{1}{6}, \frac{1}{6}, 0, 0, \frac{1}{2}, \frac{1}{2}, 0, 0, 0, \frac{1}{2})$ is a feasible solution. 
Looking at the degrees, one can see that it is an optimal solution and its optimal value is $\frac{11}{6}$. 
Therefore, $\lct_0(\a)=\frac{11}{6}$. 
We remark that $(\frac{1}{3}, 0, 0, \frac{1}{2}, 0, 0, 0, 0, 0, 1)$ is another optimal solution but it does not satisfy the assumption of Proposition \ref{criterion}. 
\end{eg}

\begin{small}
\begin{acknowledgement}
The authors are indebted to Holger Brenner for pointing out a mistake in a previous version of the article and to Karl Schwede for careful reading of the article and helpful comments. 
They are grateful to Jason Howald, Kyungyong Lee, Irena Swanson, Zach Teitlar, Howard Thompson, Kei-ichi Watanabe and Cornelia Yuen for valuable conversations. 
They would also like to acknowledge useful comments from an anonymous referee. 
Our project began at the AIM workshop ``Integral closure, Multiplier ideals and Cores" in December, 2006. We would like to thank AIM for providing a stimulating environment. 
The second author was partially supported by Grant-in-Aid for Young Scientists (B) 20740019 from JSPS and by Program for Improvement of Research Environment for Young Researchers from SCF commissioned  by MEXT of Japan. 
\end{acknowledgement}
\end{small}


\begin{thebibliography}{99}
\bibitem{BMS} 
M.~Blickle, M.~Musta\c{t}\v{a} and K.~E.~Smith, 
Discreteness and rationality of F-thresholds,
Michigan Math. J. \textbf{57} (2008), 43--61. 

\bibitem{dFM} 
T.~de~Fernex and M.~Musta\c{t}\v{a}, 
Limits of log canonical thresholds,
arXiv:0710.4978, to appear in Ann. Sci. Ecole Norm. Sup. 

\bibitem{ES} 
D.~Eisenbud and B.~Strumfels, 
On binomial ideals,
Duke Math. J. \textbf{84} (1996), 1--45. 

\bibitem{HY} 
N.~Hara and K.~Yoshida, 
A generalization of tight closure and multiplier ideals, 
Trans. Amer. Math. Soc. \textbf{355} (2003), 3143--3174.

\bibitem{Hi} 
H.~Hironaka, 
Resolution of singularities of an algebraic variety over a field of characteristic zero I, II, 
Ann. Math. (2) \textbf{79} (1964), 109--203; ibid. (2) \textbf{79} (1964), 205--326.  

\bibitem{Ho1} 
J.~Howald, 
Multiplier ideals of monomial ideals,
Trans. Amer. Math. Soc \textbf{353} (2001), 2665--2671.

\bibitem{Ho2}  
J.~Howald, 
Multiplier ideals of sufficiently general polynomials,
preprint 2003, arXiv:\linebreak math.AG/0303203.

 \bibitem{L} 
 R.~Lazarsfeld, 
 Positivity in Algebraic Geometry II.  Ergebnisse der Mathematik und ihrer Grenzgebiete. 3. Folge, 
 A Series of Modern Surveys in Mathematics, Vol. \textbf{49}, Springer-Verlag, Berlin (2004)

\bibitem{MTW} 
M.~Musta\c{t}\v{a}, S.~Takagi and K.-i.~Watanabe, 
F-thresholds and Bernstein-Sato polynomials, 
European Congress of Mathematics, 341--364, Eur. Math. Soc., Z\"{u}rich, 2005.

\bibitem{R}
N.~R\"ohrl, 
Binomial regular sequences and S-matrices, Diplomarbeit, Universit\"at Regensburg, 1998, available at {\sf http://www.iadm.uni-stuttgart.de/LstAnaMPhy/Roehrl/}.

\bibitem{Sh} 
T.~Shibuta, 
An algorithm for computing multiplier ideals, 
preprint 2008, arXiv:0807.4302. 

\bibitem{Ta}
S.~Takagi, 
Formulas for multiplier ideals on singular varieties,
Amer. J. Math. \textbf{128} (2006), 1345--1362.

\bibitem{TW} S.~Takagi and K.-i.~Watanabe, 
On F-pure thresholds, 
J. Algebra \textbf{282} (2004), no.1, 278--297.

\end{thebibliography}
\end{document}